\documentclass[12pt,a4paper,reqno]{amsart}

\usepackage{a4wide}
\makeatletter
\@namedef{subjclassname@2020}{%
	\textup{2020} Mathematics Subject Classification}
\makeatother

\usepackage{blindtext}
\usepackage[T1]{fontenc}
\usepackage[utf8]{inputenc}
\usepackage{amsmath,amsthm,amsfonts,amssymb,amscd}
\usepackage{dsfont}
\usepackage[shortlabels]{enumitem}
\usepackage{bigints}
\usepackage{pst-node}
\usepackage{tikz-cd}
\usepackage[final]{pdfpages}
 \usepackage[hidelinks]{hyperref}

\numberwithin{equation}{section}

\newtheorem{thm}{Theorem}[section]
\newtheorem{thmx}{Theorem}

\newtheorem{prop}[thm]{Proposition}
\newtheorem{cor}[thm]{Corollary}
\newtheorem{lemma}[thm]{Lemma}
\newtheorem{defn}[thm]{Definition}

\theoremstyle{remark}
\newtheorem{example}{Example}[section]

\newcommand{\R}{{\mathbb R}}
\newcommand{\N}{{\mathbb N}}
\newcommand{\Z}{{\mathbb Z}}
\newcommand{\T}{{\mathbb T}}
\newcommand{\E}{{\mathbb E}}

\newcommand{\cF}{\mathcal{F}}
\newcommand{\cW}{\mathcal{W}}
\DeclareMathOperator{\ld}{LD}
\DeclareMathOperator{\mld}{MLD}
\newcommand{\leb}{m}
\DeclareMathOperator{\sgn}{sgn}

\DeclareMathOperator{\diam}{diam}
\DeclareMathOperator{\Int}{Int}

\newcommand{\bphi}{{\bar \varphi}}
\newcommand{\bpsi}{{\bar \psi}}
\newcommand{\bDelta}{{\bar \Delta}}
\newcommand{\bpi}{{\bar \pi}}
\newcommand{\fD}{{f_\Delta}}
\newcommand{\bfD}{{\bar f_\Delta}}
\newcommand{\bY}{{\bar Y}}
\newcommand{\bmu}{{\bar \mu}}

\thanks{JFA and JSM are partially supported by CMUP (UID/MAT/00144/2025) and PTDC/MAT-PUR/4048/2021, which are funded by FCT (Portugal) with national (MEC) and European structural funds through the program FEDER, under the partnership agreement PT2020. JFA is also supported by Royal Society Wolfson Visiting Fellowship RSWVF$\backslash25\backslash$R$1\backslash1006$. JSM is also supported by the FCT doctoral scholarship 2021.07090.BD.
The research of IM was supported in part by FAPESP grant number 2024/22093-5 at Instituto de Matem\'{e}tica e Estat\'{i}stica/IME/USP, S\~{a}o Paulo, and by the Sydney Mathematical Research Institute (SMRI-2026).
}

\begin{document}

	\title{Martingale Methods for Maximal Large Deviations and Young Towers}

	\author[J. F. Alves]{Jos\'{e} F. Alves}
	\address{Jos\'{e} F. Alves\\ Centro de Matem\'{a}tica da Universidade do Porto\\ Rua do Campo Alegre 687\\ 4169-007 Porto\\ Portugal\\ and Department of Mathematical Sciences\\
Loughborough University\\
Loughborough
LE11 3TU\\
United Kingdom}
	\email{jfalves@fc.up.pt} \urladdr{http://www.fc.up.pt/cmup/jfalves}

	\author[J. S. Matias]{Jo\~{a}o S. Matias}
	\address{Jo\~{a}o S. Matias\\
		Centro de Matem\'{a}tica da Universidade do Porto\\ Rua do Campo Alegre 687\\ 4169-007 Porto\\ Portugal.}
	\email{up201504959@fc.up.pt} 
	
	\author[I. Melbourne]{Ian Melbourne}
	\address{Ian Melbourne\\
                Mathematics Institute\\ University of Warwick\\ Coventry CV4 7AL\\ United Kingdom.}
	\email{I.Melbourne@warwick.ac.uk} 
	
	\begin{abstract}
		We develop a martingale approximation framework yielding quantitative maximal large deviations estimates for invertible dynamical systems. From suitable decay of correlations, we deduce these estimates and, as an application, we obtain Young structures with matching recurrence tails for partially hyperbolic diffeomorphisms with mostly expanding central direction.
In a second application, we prove maximal large deviation estimates for systems modelled by Young towers with subexponential contraction and expansion. Many examples of slowly mixing billiards are covered by this result.
	\end{abstract}

\maketitle

\section{Introduction}

One of the cornerstones of ergodic theory is \emph{Birkhoff’s Ergodic Theorem}, which ensures that if 
$f\colon M\to M$ is an ergodic measure-preserving transformation on a probability space $(M,\mu)$ and 
$\varphi\colon M\to\R$ is integrable, then the time averages, or normalized Birkhoff sums, converge almost surely to the spatial averages:
\begin{equation*}
	\lim\limits_{n\to\infty}\frac{1}{n}\sum_{j=0}^{n-1}\varphi\circ f^j=\int\varphi \, d\mu.
\end{equation*}

This naturally raises the interesting question of how quickly this convergence takes place, that is, how fast the \emph{large deviations} at time~$n$,
\begin{equation*}
	\ld(\varphi,\varepsilon,n)=\mu\left(\left|\frac{1}{n}\sum_{j=0}^{n-1}\varphi\circ f^j - \int\varphi \, d\mu\right|>\varepsilon\right),
\end{equation*}
tend to zero as $n\to\infty$. Such estimates provide precise information about the ``tails'' of the distribution of Birkhoff sums and are closely related to the rate of decay of correlations and limit theorems. 

Regarding this speed of convergence, the classical problem was obtaining the \emph{large deviation principle}, i.e., determining a rate function $c(\varepsilon)$ such that
\begin{equation*}
	\lim\limits_{n\to+\infty}\frac{1}{n}\log 
\ld(\varphi,\varepsilon,n)
=-c(\varepsilon).
\end{equation*} 
Large deviation principles have been obtained for several classes of systems with hyperbolic behaviour. In the uniformly hyperbolic (Axiom A) setting, the theory is well understood, both for discrete- and continuous-time systems \cite{K90b, L90, OP88, Y90}. For nonuniformly expanding/hyperbolic maps and flows, see for example \cite{A07,AP06,ChungRiveraTakahasi19,CommanRivera11,MN08}.

More recently, several works have established a precise quantitative link between the decay of correlations and large deviations (LD). In particular, in \cite{AF19, AFL11, CDM22, M09, MN08} the authors showed that LD estimates can be derived from information on the decay of correlations, with comparable asymptotic rates. 
In the remainder of the introduction, we focus on polynomial rates for ease of exposition, but stretched exponential rates are covered equally from Section~\ref{sec:state} onwards.

Given observables $\varphi,\psi\colon M\to\R$, their \emph{correlation function} is defined by
\begin{equation*}\label{cor}
	\rho_{\varphi, \psi}(n) = \left|\int \varphi (
	\psi\circ f^n) \, d\mu - \int \varphi \, d\mu \int \psi \, d\mu\right|,
\quad n\geqslant 1.
\end{equation*}
If there exists a constant $\beta>0$ such that
\[
	\rho_{\varphi, \psi}(n)\lesssim n^{-\beta},
\]
then we speak of \emph{polynomial decay of correlations}.\footnote{Here $ a_n \lesssim b_n$ means that there exist constants $C>0$, $n_0\geqslant 1$ such that $a_n\leqslant C b_n$ for all $n\geqslant n_0$.}
For noninvertible maps, the link between polynomial decay of correlations and polynomial LD was revealed in~\cite{M09,MN08}.
An important consequence of quantitative LD estimates appeared in \cite{AFL11}, where such bounds were used to construct (one-sided) Young towers with polynomial return-time tails. In this way, they obtained a converse to Young~\cite[Theorem~3]{Y99}. 

Subsequently, using \emph{maximal large deviations (MLD)}, 
\begin{equation*}
	\mld(\varphi, \varepsilon, n)
	= \mu\!\left(\sup_{k\geqslant n}\left|\frac{1}{k}\sum_{j=0}^{k-1}\varphi\circ f^j-\int\varphi\,d\mu\right|>\varepsilon\right),
\end{equation*}
the authors in \cite{BS23} were able to obtain sharper estimates
for the return-time tails of the Young towers constructed in~\cite{AFL11}. 
Clearly, the decay rate of $\mld(\varphi, \varepsilon, n)$ implies a corresponding decay rate for $\ld(\varphi, \varepsilon, n)$.
As pointed out in \cite{BS23}, 
the converse automatically holds in the exponential and stretched exponential cases, but not in the polynomial case, leading to the improvement in the result of~\cite{AFL11}.

The results described above reveal a deep link for nonuniformly expanding maps between decay of correlations, MLD and geometric structures underlying the 
dynamics. 
One of our main goals in this work is to generalise such links to partially hyperbolic systems with contracting directions.

\subsection*{Nonuniformly expanding maps}

We now describe more precisely the existing results in the noninvertible setting.
Suppose that $f$ is a piecewise smooth $C^{1+\eta}$ map 
on a finite-dimensional Riemannian manifold $M$ and that $\mu$ is an $f$-invariant ergodic probability measure.
Let $\varphi\colon M\to\R$ be a fixed H\"older continuous observable.

\begin{itemize}

\parskip=3pt
\item[(1)] By Young~\cite{Y99}, if $f$ is modelled by a (one-sided) Young tower with tails $n^{-(\beta+1)}$, then $\rho_{\varphi,\psi}(n)\lesssim \|\psi\|_\infty n^{-\beta}$ for all $\psi\in L^\infty(M)$. This rate is optimal~\cite{Gouezel04,Sarig02}.
\item[(2)] By Melbourne \& Nicol~\cite{M09,MN08}, if 
$\rho_{\varphi,\psi}(n)\lesssim \|\psi\|_\infty n^{-\beta}$ for all $\psi\in L^\infty(M)$, then $\ld(\varphi,\varepsilon,n)\lesssim n^{-\beta}$ and this rate is optimal.
\end{itemize}

Now suppose in addition that all Lyapunov exponents of $f$ are positive. 
Then the circle of implications between Young tower, decay of correlations and LD can be completed.
\begin{itemize}
\item[(3)] By Alves \emph{et al.}~\cite{AFL11,ALP05}, if 
$\ld(\varphi,\varepsilon,n)\lesssim n^{-\beta}$ for all H\"older observables $\varphi$, then $f$ is modelled by a Young tower with tails $n^{-(\beta-1)}$.
\end{itemize}

As pointed out in Alves \emph{et al.}~\cite{AFL11}, the combination of (2) and (3) implies that if correlations for H\"older observables against $L^\infty$ observables decay at rate $n^{-\beta}$, then there exists a Young tower with polynomial tails $n^{-(\beta-1)}$. Moreover, they treated maps with critical points and singularities. In this way, they obtained a converse to the result of Young~\cite{Y99} in~(1).
Note however that this is not quite a perfect converse to (1) since there is a discrepancy of $2$ in the polynomial degree. The situation was improved in~\cite{BS23} who proved 
\begin{itemize}

\parskip=3pt
\item[(2$'$)] If
$\rho_{\varphi,\psi}(n)\lesssim \|\psi\|_\infty n^{-\beta}$ for all $\psi\in L^\infty(M)$, then $\mld(\varphi,\varepsilon,n)\lesssim n^{-\beta}$.
\item[(3$'$)] If 
$\mld(\varphi,\varepsilon,n)\lesssim n^{-\beta}$ for all H\"older observables $\varphi$, then $f$ is modelled by a Young tower with tails $n^{-\beta}$.
\end{itemize}
This still does not provide a perfect converse to (1), but the discrepancy in the polynomial degree has been reduced to $1$.
For polynomial return-time tails, this is currently the best result available.

\subsection*{Nonuniformly hyperbolic diffeomorphisms}

Now suppose that $f\colon M\to M$ is a $C^{1+\eta}$ diffeomorphism
on a finite-dimensional Riemannian manifold $M$ and that $\mu$ is an $f$-invariant ergodic probability measure.
A longstanding problem has to been to obtain satisfactory generalisations of the noninvertible results above in the invertible setting.
For background on Young towers~\cite{Y98,Y99}; see Section~\ref{sec:towers}.

An analogue of (1) is due to Gou\"ezel~\cite{GouezelPC} using ideas from~\cite{CG12}, and is written down in~\cite[Appendix]{MT14} and~\cite{KKM19}:
If $f$ is modelled by a (two-sided) Young tower with tails $n^{-(\beta+1)}$, then $\rho_{\varphi,\psi}(n)\lesssim \|\phi\|_{C^\eta}\|\psi\|_{C^\eta} n^{-\beta}$ for all H\"older observables $\varphi,\,\psi$. 
Unfortunately, it is well-known (see for example~\cite[Remark~6]{MT02})
that it is
no longer possible to obtain decay of correlations against all $\psi\in L^\infty(M)$. Moreover, decay of correlations against all H\"older observables $\psi$ seems insufficient for proceeding further. A stronger analogue of (1) was obtained by Demers \emph{et al.}~\cite{DMN20}. Let $L^\infty(\cF_0)$ be the space of $L^\infty$ observables that are constant along local stable leaves (the notation will be explained later on in Section~\ref{sec:state}). 
\begin{itemize}
\item[(1$''$)] By~\cite{DMN20}, if $f$ is modelled by a (two-sided) Young tower with tails $n^{-(\beta+1)}$, then $\rho_{\varphi,\psi}(n)$ decays at rate $n^{-\beta}$ for all $\varphi$ H\"older and $\psi\in L^\infty(\cF_0)$. Moreover, on average, stable leaves are contracted under $\varphi\circ f^n$ at rate $n^{-\beta}$.
\end{itemize}
The first main result in this paper, Theorem~\ref{thm:mld} below, provides the desired analogue to~(2$'$):
\begin{itemize}
\item[(2$''$)] Let $\varphi\in L^\infty(M)$.
Suppose that $\rho_{\varphi,\psi}(n)$ decays at rate $n^{-\beta}$ for all $\psi\in L^\infty(\cF_0)$ and on average, stable leaves are contracted under $\varphi\circ f^n$ at rate $n^{-\beta}$.
Then $\mld(\varphi,\varepsilon,n)\lesssim n^{-\beta}$.
\end{itemize}
Suppose in addition that 
$\mu$ is an $SRB$ measure supported on a partially hyperbolic set with exponentially contracting stable directions and with all Lyapunov exponents along $E^{cu}$ positive.
\begin{itemize}
\item[(3$''$)] 
By~\cite[Theorem~A]{AM24}, 
if $\mld(\varphi,\varepsilon,n)\lesssim n^{-\beta}$ for all H\"older observables $\varphi$, then $f$ is modelled by a Young tower with tails $n^{-\beta}$.
\end{itemize}

It is worth noting that   \cite[Theorem~D]{ADLP17}  plays a key role  in the proof of \cite[Theorem~A]{AM24}.
In this way, for partially hyperbolic systems, we are able to obtain the desired converse (again up to the discrepancy of $1$ in the polynomial degree) to (1$''$).
This result, together with a version for stretched exponential rates is stated in Theorem~\ref{thm:young}.

\subsection*{MLD for systems with subexponential contraction and expansion}

The afore-mentioned results on polynomial LD/MLD hold in particular for noninvertible systems modelled by (one-sided) Young towers with polynomial tails. Again, an important longstanding question was to prove such results for invertible systems modelled by (two-sided) Young towers with polynomial tails. This is easily resolved when there is exponential contraction along stable leaves: any H\"older observable $\varphi$ can be lifted to a sum $\hat\varphi+b$ where $\hat\varphi$ is dynamically H\"older on the tower and constant along stable leaves while $b$ is an $L^\infty$ coboundary~\cite[Lemma~3.2]{MN05}.  In this way, MLD for $\varphi$ reduces to MLD for $\hat\varphi$ which holds by the results for one-sided towers.

Unfortunately the assumption about exponential contraction along stable leaves is somewhat hidden in~\cite[Section~3, Assumption (A2)(i)]{MN05}, and hence~\cite{M09,MN08}, causing some confusion (including to the authors of~\cite{M09,MN08} who incorrectly applied their results to certain billiard examples~\cite[Examples~1.4 and~1.5]{MN08}). In general, LD/MLD for H\"older observables for invertible systems modelled by polynomial two-sided Young towers is previously unproved. This is resolved in our final main result, Theorem~\ref{thm:Y99}.
As a consequence, in Section~\ref{sec:ex},
 we obtain new results on LD/MLD for numerous examples of slowly mixing billiards (including the incorrectly-argued examples in~\cite{M09,MN08}). 

\bigskip

The remainder of this paper is organised as follows.
Our main results, Theorems~\ref{thm:mld},~\ref{thm:young} and~\ref{thm:Y99} are stated precisely in Section~\ref{sec:state}.
In Section~\ref{sec:MA}, we prove a purely probabilistic result, Theorem~\ref{ld}, showing how to prove MLD via martingale approximations.
Theorem~\ref{thm:mld} is then proved in Section~\ref{sec:DC}.
Theorem~\ref{thm:young} is proved in Section~\ref{sec:Young} as an application of Theorem~\ref{thm:mld} together with~\cite[Theorem~A]{AM24}.
Theorem~\ref{thm:Y99} is proved in Section~\ref{sec:Y99} as a second application of Theorem~\ref{thm:mld}.
In Section~\ref{sec:ex}, we treat a number of examples using Theorem~\ref{thm:Y99}.

\section{Statement of the main results}
\label{sec:state}

For the statement of our first main result, we 
suppose that $f\colon M\to M$ is an invertible ergodic measure-preserving transformation defined on a probability space $(M,\mu)$. Let $\cW$ be a countable partition of $M$ into measurable sets and let $\cF_0$ be the $\sigma$-algebra generated by $\cW$.
 
Throughout this section, we consider two types of rate sequence $r(n)$:
\begin{defn} \label{defn:r}
 		\textbf{Stretched exponential rates:}
$r(n)= e^{-\tau n^\omega}$ where
$ \tau > 0 $ and ${ 0 < \omega \leqslant 1 }$ are constants;
in this case we set $r'(\varepsilon,n)=\exp\{-\tau' \varepsilon^{\omega}n^{\frac{\omega}{2}}\}$ for $\tau'>0$ chosen sufficiently small.

 		\noindent \textbf{Polynomial rates:}
$r(n)= n^{-\beta}$ for some
constant $ \beta > 0 $; in this case,
fix any $p>\max\left\{2, 2\beta\right\}$, and define 
 		$r'(\varepsilon,n)=\varepsilon^{-p}n^{-\beta}$.
\end{defn}

 \begin{thmx}\label{thm:mld}
 Assume that $f^{-1}\cF_0\subset\cF_0$ and let
$r,\,r'$ be as in Definition~\ref{defn:r}.
 Let $\varphi\colon M\to\R$ be an $L^\infty$ mean zero observable.
Suppose that there is a constant $C>0$ such that
 		\begin{enumerate}[\em(i)]

\parskip=3pt
 			\item $\rho_{\varphi,\psi}(n)\leqslant C\left\lVert\psi\right\rVert_\infty r(n)$, for all $\psi\in L^\infty(\cF_0)$, $n\geqslant1$;

\vspace{.5ex}
 			\item $\sum_{W\in f^n\cW} \mu(W)\diam\varphi(W)\leqslant Cr(n)$,
for all $n\geqslant1$.
 		\end{enumerate}
Then there exists a constant $C'>0$ such that
 		$$
 		\mld(\varphi,\varepsilon,n)\leqslant C'r'(\varepsilon,n)
 		\quad\text{for all $n\geqslant1$, $\varepsilon>0$}.
 		$$
 \end{thmx}

Our first application of Theorem~\ref{thm:mld} is in the partially hyperbolic setting with contracting directions. We give sufficient conditions based on rates of decay of correlations for the existence of a Young tower with specified tails.

Let $f\colon M\to M$ be a diffeomorphism defined on a compact Riemannian manifold $(M,d)$.
We say that a compact invariant set $X \subset M$ is \emph{partially hyperbolic set} if there exists a $Df$-invariant splitting 
\begin{equation} \label{eq:split}
T_X M = E^s \oplus E^{cu},
\end{equation}
with $\dim E^s\geqslant1$, $\dim E^{cu}\geqslant1$, together with constants $C>0$, $0 < \lambda < 1$, such that 
\begin{itemize}
	\item $E^s$ is \emph{uniformly contracting}: $\lVert Df|_{E_x^s} \rVert \leqslant C \lambda$, for all $x \in X$;
	\item $E^{cu}$ is \emph{dominated} by $E^s$: 
	\(
	\lVert Df|_{E_x^s} \rVert \cdot \lVert Df^{-1}|_{E_{fx}^{cu}} \rVert \leqslant C \lambda\), for all \(x \in X.
	\)
\end{itemize}
We refer to $E^s$ as the \emph{stable} subbundle and to $E^{cu}$ as the \emph{centre-unstable} subbundle. 
The stable bundle $E^s$ integrates to a lamination $\cW^s$ consisting of stable leaves, and the $\sigma$-algebra $\cF_0$ generated by $\cW^s$ satisfies $f^{-1}\cF_0\subset\cF_0$.

Recall that an observable $\varphi \colon M \to \R$ is said to be $\eta$-H\"older continuous, for some $\eta \in(0,1]$, if there exists a constant $C > 0$ such that
$$
\left| \varphi(x) - \varphi(y) \right| \leqslant C d(x, y)^\eta, \quad \text{for all $ x, y \in M$}.
$$
Let $C^\eta(M)$ denote the space of all $\eta$-H\"older continuous observables. The H\"older norm of an observable $\varphi \in C^\eta$ is defined as
$$
\lVert \varphi \rVert_{C^\eta} := \lVert \varphi \rVert_\infty + \sup_{x \neq y} \frac{\left| \varphi(x) - \varphi(y) \right|}{d(x, y)^\eta}.
$$
The following theorem provides quantitative conditions under which the statistical properties of the system ensure the existence of a Young tower.

\begin{thmx}\label{thm:young}
	Let $f\colon M\to M$ be a $C^{1+\eta}$ diffeomorphism and let $X\subset M$ be a partially hyperbolic set. Suppose that $f$ admits an ergodic $SRB$ measure $\mu$ supported on $X$ with all Lyapunov exponents along $E^{cu}$ positive.  Let $r,\,r'$ be as in Definition~\ref{defn:r}.

		Assume that there exists a constant $ C > 0 $ such that
		$$\rho_{\varphi,\psi}(n)\leqslant C\left\lVert\varphi\right\rVert_{C^\eta} \left\lVert\psi\right\rVert_\infty r(n), \quad \text{for all $\varphi \in C^\eta$, $\psi\in L^\infty(\cF_0)$ and $n\geqslant1$}.$$ 
		Then $f$ is modelled by a (two-sided) Young tower with 
$$
\leb_\gamma\left\{R>n\right\}\lesssim r'(1,n), \quad\text{for all $\gamma\in\Gamma^u$}.
$$
\end{thmx} 

Our second application of Theorem~\ref{thm:mld} is to obtain MLD for dynamical systems modelled by a Young tower with subexponential tails.

\begin{thmx} \label{thm:Y99} Suppose that $f\colon M\to M$ is an invertible dynamical system modelled by a (two-sided) Young tower with return time function $R$.
Let $\varphi\in C^\eta(M)$ and let $\varepsilon>0$.
 \begin{enumerate}[\rm(a)]
                \item\label{Y99exp} Assume that there exist constants $ \tau > 0 $ and $ 0 < \omega \leqslant 1 $ such that
                $\leb\{R>n\}\lesssim e^{-\tau n^{\omega}}$.
Let $\omega'=\omega/(1+\omega)$. 
Then
                $\mld(\varphi,\varepsilon,n)\lesssim \exp\{-\tau' n^{\frac{\omega'}{2}}\}$ for some $\tau'\in(0,\tau)$.
\item\label{Y99pol}
 If there exists $ \beta > 0 $ such that
                $\leb\{R>n\}\lesssim n^{-(\beta+1)}$,
then 
                $\mld(\varphi,\varepsilon,n)\lesssim n^{-\beta}.  $
\end{enumerate}
\end{thmx}

	\section{MLD via Martingale Approximations}\label{sec:MA}
	
The goal of this section is to obtain LD/MLD for invertible maps via martingale approximations techniques in a purely probabilistic setting.
Such martingale approximations were originally due to Gordin~\cite{G69a,HallHeyde80}. As in~\cite{DMN20},
we use a refinement due to Dedecker \emph{et al.}~\cite{DMP13}.

We consider both polynomial MLD and (stretched) exponential LD.
(As already pointed out in~\cite{BS23}, (stretched) exponential MLD is not a useful notion.)

Suppose that $f\colon M\to M$ is an invertible measure-preserving transformation on the probability space $(M,\cF,\mu)$.

 \begin{thm}\label{ld}
Assume that $\cF_0\subset\cF$ is a $\sigma$-subalgebra satisfying $f^{-1}\cF_0\subset\cF_0$.
 Let ${\varphi\colon M\to\R}$ be an $L^\infty$ mean zero observable.
 	\begin{enumerate}[\rm(a)]
 		\item \label{main_exp} Assume that there exist constants $ C> 0 $, $ \tau > 0 $ and $ {0 < \omega \leqslant 1 }$ such that, for all $p\geqslant1$, $n\geqslant 1$,
 		\begin{enumerate}[\rm(i)]
 			\item $ \left\lVert \E\left( \varphi \circ f^{-n} \mid \cF_0 \right) \right\rVert_p 
 			\leqslant C e^{-\frac{\tau}{p} n^{\omega}} $;
 			\item $ \left\lVert \E\left( \varphi \circ f^n \mid \cF_0 \right) - \varphi \circ f^n \right\rVert_p 
 			\leqslant Ce^{-\frac{\tau}{p} n^{\omega}}. $
 		\end{enumerate}
 		Then, there exists $\tau'>0$ such that
 		$$
 		\ld(\varphi,\varepsilon,n)\leqslant 2\exp\{-\tau' n^{\frac{\omega}{2}}\varepsilon^{\omega}\},
\quad\text{for all $n\geqslant1$, $\varepsilon>0$}.
 		$$
 		
 		\item \label{main_pol} Assume that there exist constants $C> 0 $, $ \beta > 0 $ and $p>\max\left\{2,2\beta\right\}$ such that, for all $n\geqslant1$,
 		\begin{enumerate}[\rm(i)]
 			\item $ \left\lVert \E\left( \varphi \circ f^{-n} \mid \cF_0 \right) \right\rVert_p 
 			\leqslant C n^{-\beta/p} $;
 			\item $ \left\lVert \E\left( \varphi \circ f^n \mid \cF_0 \right) - \varphi \circ f^n \right\rVert_p 
 			\leqslant C n^{-\beta/p}. $
 		\end{enumerate}
 		Then, there exists $K>0$ such that 
 		$$
 		\mld(\varphi,\varepsilon,n)\leqslant K\varepsilon^{-p}n^{-\beta},
\quad\text{for all $n\geqslant1$, $\varepsilon>0$}.
$$
 	\end{enumerate}
 \end{thm}

 The proof of Theorem~\ref{ld} is presented in the next two subsections.

	\subsection{Stretched Exponential Case}\label{sec:exp}
	
	 In this subsection, we establish Theorem~\ref{ld}(a). 
We assume throughout that $\varphi\colon M\to\R$ is a mean zero $L^\infty$ observable
satisfying the hypotheses of Theorem~\ref{ld}(a).
	For $n\geqslant1$, we define the Birkhoff sum
	\begin{equation*}
		\varphi_n=\sum_{k=0}^{n-1}\varphi\circ f^k.
	\end{equation*}	
	
	\begin{lemma}\label{lemma2}
	There exists a constant $C>0$ such that 
	\begin{equation*}
		\Big\lVert\max_{k\leqslant n}\left|\varphi_k\right|\Big\rVert_p\leqslant Cp^{\frac{1}{\omega}}n^{\frac{1}{2}},
\quad\text{for all $n\geqslant1$ and $p>2$.}
	\end{equation*}
	\end{lemma}

	\begin{proof}
		Let $A_n=\sum_{j=1}^n\varphi\circ f^{-j}$. Notice that for every $k\leqslant n$,
		\begin{equation*}
			\varphi_k=\left(A_n-A_{n-k}\right)\circ f^n.
		\end{equation*}
		Thus, we can deduce that
		\begin{equation}\label{eq.phik}
			\Big\lVert\max_{k\leqslant n}\left|\varphi_k\right|\Big\rVert_p=\Big\lVert\max_{k\leqslant n}\left|A_n-A_{n-k}\right|\Big\rVert_p\leqslant 2\Big\lVert\max_{k\leqslant n}\left|A_k\right|\Big\rVert_p,
		\end{equation}
		which means that it suffices to estimate $\big\lVert\max_{k\leqslant n}\left|A_k\right|\big\rVert_p$ to complete the proof.
		
		In order to do so, we apply \cite[Corollary~3.9]{DMP13}.
(Note that $T$ in~\cite{DMP13} corresponds to our $f^{-1}$.)
For $r\geqslant1$ and $p>2$, this gives
		\begin{equation}\label{DMP}
		\begin{aligned}
				\int \max_{k\leqslant 2^r}\left|A_k\right|^p \, d\mu \leqslant 2^{\frac{rp}{2}}\int \left|\varphi\right|^p \, d\mu &+ 2^{\frac{rp}{2}}\left( \sum_{j=0}^{r-1}2^{-\frac{j}{2}} \left\lVert\E\left(A_{2^j}|\cF_0\right)\right\rVert_p\right)^p\\
				&+ 2^{\frac{rp}{2}}\left(\sum_{j=1}^{r}2^{-\frac{j}{2}} \big\lVert A_{2^j}-\E(A_{2^j}|f^{2^j}\cF_0)\big\rVert_p\right)^p.
		\end{aligned}
		\end{equation}
		We estimate each of the three terms individually. 
In what follows we use $C_0$ to denote a positive constant whose value is independent of $n$, $p$ and $r$. The value of $C_0$ may change from line to line.

For the first term, setting $C_0=\lVert\varphi\rVert_\infty 2^{-1/\omega}$, we have 
		\begin{equation}\label{eq.s1}
			\lVert \varphi\rVert_p\leqslant \lVert\varphi\rVert_\infty = C_02^{1/\omega}\leqslant C_0p^{1/\omega}.
		\end{equation}	

		For the second term in~\eqref{DMP}, 
		\begin{equation}\label{sum2}
			\begin{aligned}
				 \sum_{j=0}^{r-1} 2^{-\frac{j}{2}}\left\lVert \E\left(A_{2^j}|\cF_0\right)\right\rVert_p
				&\leqslant\sum_{j=0}^{r-1} 2^{-\frac{j}{2}}\sum_{k=1}^{2^j}\left\lVert\E\left(\varphi\circ f^{-k}|\cF_0\right)\right\rVert_p\\
				&=\sum_{k=1}^{2^{r-1}}\sum_{j=\lceil \log_2 k\rceil}^{r-1} 2^{-\frac{j}{2}}\left\lVert \E\left(\varphi\circ f^{-k}|\cF_0\right)\right\rVert_p\\
				&\leqslant (1-2^{-\frac12})^{-1}\sum_{k=1}^{2^{r-1}} \left\lVert \E\left(\varphi\circ f^{-k}|\cF_0\right)\right\rVert_p
				\leqslant C_0\sum_{k=1}^{\infty}e^{-\frac{\tau}{p} k^\omega},
			\end{aligned}
		\end{equation}
		where in the final inequality we have used hypothesis~(i) of Theorem~\ref{ld}(a). As in the proof of \cite[Lemma~5]{AF19}, we deduce that		
		\begin{equation}\label{sum2exp}
			\begin{aligned}
				\sum_{k=1}^{\infty}e^{-\frac{\tau}{p} k^\omega}\leqslant \int_{0}^{\infty}e^{-\frac{\tau}{p}t^\omega}\, dt
				&=\frac{1}{\omega}\left(\frac{p}{\tau}\right)^{\frac{1}{\omega}}\int_{0}^{\infty}e^{-s}s^{\frac{1}{\omega}-1}\, ds
				=C_0p^{\frac{1}{\omega}},
			\end{aligned}
		\end{equation}
		where we have performed the change of variables $s=\frac{\tau}{p}t^\omega$.		
		Hence,
		\begin{equation}\label{eq.s2}
			\sum_{j=0}^{r-1}2^{-\frac{j}{2}} \left\lVert \E\left(A_{2^j}|\cF_0\right)\right\rVert_p\leqslant C_0 p^\frac{1}{\omega}.
		\end{equation}		
	To estimate the third term in~\eqref{DMP}, we start by noticing that
	\begin{equation}\label{sum3exp}
		\begin{aligned}
			\big\lVert A_{2^j}-\E\big(A_{2^j}|f^{2^j}\cF_0\big)\big\rVert_p
			&\leqslant \sum_{k=1}^{2^j}\big\lVert \varphi\circ f^{-k}-\E\big(\varphi\circ f^{-k}|f^{2^j}\cF_0\big)\big\rVert_p\\
			&=\sum_{k=1}^{2^j}\big\lVert \varphi\circ f^{2^j-k}-\E\big(\varphi\circ f^{2^j-k}| \cF_0 \big)\big\rVert_p\\
			&=\sum_{k=0}^{2^j-1}\big\lVert \varphi\circ f^{k}-\E\big(\varphi\circ f^{k}| \cF_0 \big)\big\rVert_p.
		\end{aligned}
	\end{equation}	
	Proceeding analogously to the arguments in~\eqref{sum2} and~\eqref{sum2exp}, and using hypothesis~(ii) of Theorem~\ref{ld}(a), we obtain
		\begin{equation}\label{eq.s3}
			\begin{aligned}
				\sum_{j=1}^{r}2^{-\frac{j}{2}} \big\lVert A_{2^j}-\E\big(A_{2^j}|f^{2^j}\cF_0\big)\big\rVert_p&\leqslant C_0 p^\frac{1}{\omega}.
			\end{aligned}
		\end{equation}
		Substituting~\eqref{eq.s1},~\eqref{eq.s2} and~\eqref{eq.s3} into~\eqref{DMP}, we conclude that
		\begin{equation*}
				\Big\lVert \max_{k\leqslant 2^r}\left|A_k\right|\Big\rVert_p\leqslant 2^{\frac{r}{2}}C_0 p^{\frac{1}{\omega}}.
		\end{equation*}
		Choosing $r\geqslant1$ such that $2^{r-1}<n\leqslant2^r$, we deduce that
		\begin{equation*}
			\begin{aligned}
				\Big\lVert\max_{k\leqslant n}\left|A_k\right|\Big\rVert_p&\leqslant\Big\lVert\max_{k\leqslant 2^r}\left|A_k\right|\Big\rVert_p\leqslant 2^{\frac{r}{2}}C_0p^{\frac{1}{\omega}}\leqslant C_0p^{\frac{1}{\omega}}n^{\frac{1}{2}}.
			\end{aligned}
		\end{equation*}		
By~\eqref{eq.phik}, this completes the proof.
	\end{proof}
	
	The next result extends the previous estimate to the range $p\geqslant\omega$.
	
	\begin{cor}\label{cor3} 
	There exists a constant $C>0$ such that 
	\begin{equation*}
		\Big\lVert\max_{k\leqslant n}\left|\varphi_k\right|\Big\rVert_p\leqslant Cp^{\frac{1}{\omega}}n^{\frac{1}{2}},
\quad\text{for all $n\geqslant1$ and $p\geqslant\omega$}.
	\end{equation*}
	\end{cor}
	
	\begin{proof}
	The estimate holds for $p\geqslant3$ by Lemma~\ref{lemma2}.
		For $\omega\leqslant p<3$, we have
		\begin{equation*}
\Big\lVert\max_{k\leqslant n}\left|\varphi_k\right|\Big\rVert_p
\leqslant\Big\lVert\max_{k\leqslant n}\left|\varphi_k\right|\Big\rVert_3
\leqslant C 3^{\frac{1}{\omega}}n^{\frac{1}{2}}\leqslant 
\big(C 3^{\frac{1}{\omega}}\omega^{-\frac{1}{\omega}}\big) p^{\frac{1}{\omega}}n^{\frac{1}{2}}.
		\end{equation*}		
	Hence the desired estimate holds for every $p\geqslant\omega$.
\end{proof}

	\begin{lemma}\label{lemma4}
		There exists $\tau'>0$ such that
			$\int \exp\{\tau' n^{-\frac{\omega}{2}}\left|\varphi_n\right|^\omega\}\, d\mu\leqslant 2$,
for all $n\geqslant1$.
	\end{lemma}
	
	\begin{proof}
		By Corollary~\ref{cor3},
		\begin{equation*}
\begin{aligned}
				\int \exp\{\tau'n^{-\frac{\omega}{2}}\left|\varphi_n\right|^\omega\}\, d\mu &=\sum_{j=0}^{\infty}\frac{\left(\tau'\right)^j n^{-j\frac{\omega}{2}}}{j!}\int\left|\varphi_n\right|^{j\omega}\, d\mu\\
				&\leqslant1+ \sum_{j=1}^{\infty}\frac{\left(\tau'\right)^j n^{-j\frac{\omega}{2}}}{j!} \left(C (j\omega)^{\frac{1}{\omega}}n^{\frac12}\right)^{j\omega}
				= 1+\sum_{j=1}^{\infty}\left(C^\omega \omega\tau'\right)^j\frac{j^j}{j!}.				
\end{aligned}
		\end{equation*}	
		Since $\left(1+\frac{1}{j}\right)^j \leqslant e$, by induction over $j$, we have $j^j\leqslant e^jj!$ for every $j\geqslant1$.	
		Therefore, 
		\begin{equation*}
			\int \exp\{\tau'n^{-\frac{\omega}{2}}\left|\varphi_n\right|^\omega\}\, d\mu \leqslant 1+ \sum_{j=1}^{\infty}\left(C^\omega \omega e\tau'\right)^j = 2
		\end{equation*}
		if we set $\tau'=\left(2C^\omega \omega e\right)^{-1}$.
	\end{proof}	

	We are finally ready to complete the proof of Theorem~\ref{ld}(a).
	\begin{proof}[Proof of Theorem~\ref{ld}(a):]
		By Markov's inequality and Lemma~\ref{lemma4},
		\begin{equation*}
			\begin{aligned}
				\ld\left(\varphi,\varepsilon,n\right)&=
	\mu\left(n^{-\frac12}\left|\varphi_n\right|>n^{\frac12}\varepsilon\right)
	=\mu\left(\exp\{\tau'n^{-\frac{\omega}{2}}\left|\varphi_n\right|^\omega\}>\exp\{\tau'n^{\frac{\omega}{2}}\varepsilon^\omega\}\right)\\
 &\leqslant \exp\{-\tau'n^{\frac{\omega}{2}}\varepsilon^\omega\}\int \exp\{\tau'n^{-\frac{\omega}{2}}\left|\varphi_n\right|^\omega\}\, d\mu
				\leqslant 2\exp\{-\tau'n^{\frac{\omega}{2}}\varepsilon^\omega\},
			\end{aligned}
		\end{equation*}
as required.
\end{proof}

	\subsection{Polynomial Case}\label{sec:pol}
	In this subsection we prove Theorem~\ref{ld}(b). 
We assume throughout that $\varphi\colon M\to\R$ is a mean zero $L^\infty$ observable
satisfying the hypotheses of Theorem~\ref{ld}(b).

 	 We begin with a lemma very similar to Lemma~\ref{lemma2}.

	\begin{lemma}\label{lemma1pol}
		There exists a constant $C>0$ such that 
		\begin{equation*}
			\Big\lVert\max_{k\leqslant n}\left|\varphi_k\right|\Big\rVert_p\leqslant C n^{1-\frac{\beta}{p}},
\quad\text{for all $n\geqslant1$ and $p>\max\left\{2, 2\beta\right\}$}.
		\end{equation*}
	\end{lemma}
	
	\begin{proof}
		As in the (stretched) exponential case, let
			$A_n=\sum_{j=1}^n\varphi\circ f^{-n}$.
		Following the strategy of the previous section, let $r\geqslant1$, $p>2$, so~\eqref{DMP} holds. 
	We estimate each term in the inequality~\eqref{DMP} individually. 
Again, $C_0$ denotes a positive constant whose value is independent of $n$ and $r$ and may change from line to line. However, this time $C_0$ depends on $p$.

For the first term, simply take 
	\begin{equation}\label{sum1pol}
		\lVert\varphi\rVert_p \leqslant C_0.
	\end{equation}	
For the second term,
	following the computations in~\eqref{sum2}, 
	\begin{equation*}
		\begin{aligned}
\sum_{j=0}^{r-1}2^{-\frac{j}{2}} \left\lVert \E\left(A_{2^j}|\cF_0\right)\right\rVert_p
			&\leqslant (1-2^{-\frac12})^{-1}\sum_{k=1}^{2^{r-1}} k^{-\frac{1}{2}}\left\lVert \E\left(\varphi\circ f^{-k}|\cF_0\right)\right\rVert_p
			 \leqslant C_0\sum_{k=1}^{2^r}k^{-\left(\frac{\beta}{p}+\frac{1}{2}\right)},
		\end{aligned}
	\end{equation*}	
		by hypothesis~(i) of Theorem~\ref{ld}(b). 
	Since $p>\max\left\{2,2\beta\right\}$, we have $\frac{\beta}{p} +\frac12<1$, and so
	\begin{equation}\label{sum2pol}
		\sum_{j=0}^{r-1}2^{-\frac{j}{2}} \big\lVert \E\big(A_{2^j}|\cF_0\big)\big\rVert_p\leqslant C_0\left(2^r\right)^{\frac{1}{2}-\frac{\beta}{p}}.
	\end{equation}
	
	Proceeding analogously and recalling~\eqref{sum3exp}, the third term in~\eqref{DMP} is estimated by
	\begin{equation}\label{sum3pol}
		\sum_{j=1}^{r}2^{-\frac{j}{2}} \big\lVert A_{2^j}-\E\big(A_{2^j}|f^{2^j}\cF_0\big)\big\rVert_p\leqslant C_0 \left(2^r\right)^{\frac{1}{2}-\frac{\beta}{p}}.
	\end{equation}
	
	Substituting~\eqref{sum1pol}, \eqref{sum2pol} and~\eqref{sum3pol} into~\eqref{DMP}, we conclude that
	\begin{equation*}
		\Big\lVert \max_{k\leqslant 2^r} \left|A_k\right|\Big\rVert_p\leqslant 2^{\frac{r}{2}}C_0\left(2^r\right)^{\frac{1}{2}-\frac{\beta}{p}}\leqslant C_0\left(2^r\right)^{1-\frac{\beta}{p}}.
	\end{equation*}
	Choosing $r\geqslant1$ such that $2^{r-1}<n\leqslant 2^r$, we deduce 
	\begin{equation*}
		\Big\lVert\max_{k\leqslant n}\left|A_k\right|\Big\rVert_p\leqslant\Big\lVert\max_{k\leqslant 2^r}\left|A_k\right|\Big\rVert_p\leqslant C_0n^{1-\frac{\beta}{p}}.
	\end{equation*}
The result follows by~\eqref{eq.phik}.
\end{proof}

	Define
$$M_n:=\sup_{k\geqslant n}\,k^{-1}\left|\sum_{j=0}^{k-1}\varphi\circ f^j\right|,\quad
n\geqslant1.
$$
We summarise the key argument in~\cite{BS23} for passing from large deviations to maximal large deviations as follows:

\begin{lemma} \label{BS}
Let $\phi\in L^\infty$, $p\geqslant1$, $a>0$. Suppose there is a constant $C>0$ such that
	\begin{equation*}
			n^{-1}\Big\lVert\max_{k\leqslant n}\left|\varphi_k\right|\Big\rVert_p\leqslant C n^{-a},
\quad\text{for all $n\geqslant1$}.
		\end{equation*}
Then there is a constant $C'>0$ such that
	\begin{equation*}
		\left\lVert M_n\right\rVert_p\leqslant C' n^{-a},
\quad\text{for all $n\geqslant1$}.
	\end{equation*}
\end{lemma}

\begin{proof}
The argument is almost identical to that presented in~\cite[pp.\ 8--9]{BS23}, 
so we just indicate the main steps.
Define
	$$
C_n:= 3 \, \frac{1}{3n}\left|\sum_{j=0}^{3n-1}\varphi\circ f^j\right|+ 4 \,\frac{1}{2n} \max_{n\leqslant k\leqslant2n-1}\left|\sum_{j=k}^{3n-1}\varphi\circ f^j\right|.
$$
Then a calculation~\cite{BS23} yields
$M_n\leqslant M_{2n}+C_n$.
Since $\lim_{n\to\infty}\lVert M_n \rVert_p=0$ by the pointwise ergodic theorem and dominated convergence theorem (notice that $\lVert M_n\rVert_\infty\leqslant \lVert\varphi\rVert_\infty$), it follows by induction\footnote{The argument in~\cite{BS23} seems unnecessarily complicated here (at least for $\varphi\in L^\infty$) and the induction works directly without introducing functions $b_n$ and $c_n$.} that
$$
\lVert M_n \rVert_p\leqslant \sum_{j=0}^\infty \lVert C_{2^jn}\rVert_p.
$$
Since
	$\max_{k\leqslant n-1}\left|\sum_{i=k}^{n-1}\varphi\circ f^i\right|\leqslant 2 \max_{k\leqslant n}\left|\varphi_n\right|$,
it follows from Lemma~\ref{lemma1pol} that
		\begin{equation*}
			n^{-1} \left\lVert \max_{k\leqslant n-1}\Big|\sum_{i=k}^{n-1}\varphi\circ f^i\Big|\right\rVert_p\leqslant 2Cn^{-a},
\quad\text{for all $n\geqslant1$.}
		\end{equation*}
Hence,
$
\lVert C_n\rVert_p \lesssim n^{-a}
$
and $\lVert M_n\rVert_p\lesssim \sum_{j=0}^\infty (2^j)^{-a}n^{-a}\lesssim n^{-a}$.
\end{proof}

	\begin{proof}[Proof of Theorem~\ref{ld}(b):]
Let $p>\max\left\{2,2\beta\right\}$.
By Lemmas~\ref{lemma1pol} and~\ref{BS},
$
\lVert M_n \rVert_p\lesssim n^{-\frac{\beta}{p}}.
$
		By Markov's inequality and Lemma~\ref{lemma4},
$$
	\mld\left(\varphi,\varepsilon,n\right)=
	\mu\left(\left|M_n\right|>\varepsilon\right)
\leqslant \varepsilon^{-p}\lVert M_n \rVert_p^p
\lesssim \varepsilon^{-p} n^{-\beta},
$$
as required.
\end{proof}

\section{Decay of Correlations and MLD}\label{sec:DC}

In this section we prove Theorem~\ref{thm:mld}.
First, we
prove Theorem~\ref{main} below, which is a simple extension of \cite[Theorem~3.1]{DMN20}. 
This result holds for both (stretched) exponential and polynomial rates of decay. For notational convenience, we denote the corresponding decay rate generically by $r(n)$, $n\geqslant1$. 

We continue to suppose that $f\colon M\to M$ is an invertible measure-preserving transformation on the probability space $(M,\mu)$.
Suppose that $\cW$ is a countable partition of $M$ into measurable sets and let $\cF_0$ be the $\sigma$-algebra generated by $\cW$.
We denote by $L^\infty(\cF_0)$ the space of essentially bounded $\cF_0$-measurable functions.

 \begin{thm}\label{main}
 Let $\varphi\colon M\to\R$ be an $L^\infty$ mean zero observable.
 	Assume that there exists a constant $C>0$ and a decay rate $r(n)$ such that 
 	\begin{enumerate}[\rm(i)]

\parskip=3pt
 		\item $\rho_{(\varphi,\psi}(n)\leqslant C\left\lVert\psi\right\rVert_\infty r(n)$, for all $\psi\in L^\infty(\cF_0)$,
 $n\geqslant1$;
	\item $\sum_{W\in f^n\cW} \mu(W)\diam\varphi(W)\leqslant C r(n)$, for all $n\geqslant1$.
 	\end{enumerate}
 	Then, for all $n\geqslant1$ and $p\geqslant1$,
 	\begin{enumerate}[\rm(a)]
 		\item\label{itA} $\left\lVert\E\left(\varphi\circ f^{-n}|\cF_0\right)\right\rVert_p\leqslant C^{\frac{1}{p}}\left\lVert\varphi\right\rVert_\infty^{1-\frac{1}{p}}r(n)^\frac{1}{p}$;
 		\item\label{itB} $\left\lVert \E\left(\varphi\circ f^n|\cF_0\right)-\varphi\circ f^n\right\rVert_p\leqslant 2C^{\frac{1}{p}}\left\lVert\varphi\right\rVert_\infty^{1-\frac{1}{p}}r(n)^{\frac{1}{p}}$.
 	\end{enumerate}
 \end{thm}

\begin{proof}
	Let $\psi=\left|\E\left(\varphi\circ f^{-n}|\cF_0\right)\right|^{p-1}\sgn\left(\E\left(\varphi\circ f^{-n}|\cF_0\right)\right)$. 
Then $\psi\in L^\infty(\cF_0)$ and $\left\lVert\psi\right\rVert_\infty\leqslant\left\lVert\varphi\right\rVert_\infty^{p-1}$. 
We recall the identities 
	\begin{equation*}
\E(\varphi\circ f^{-n}|\cF_0)\circ f^n= \E(\varphi|f^{-n}\cF_0)
\end{equation*}
and thereby
	\begin{equation*}
		\psi\circ f^n=\left|\E\left(\varphi|f^{-n}\cF_0\right)\right|^{p-1}\sgn\left(\E\left(\varphi|f^{-n}\cF_0\right)\right).
	\end{equation*}
	Using these, we obtain
	\begin{equation*}
		\begin{aligned}
			\left\lVert\E\left(\varphi\circ f^{-n}|\cF_0\right)\right\rVert_p^p&=\left\lVert\E\left(\varphi|f^{-n}\cF_0\right)\right\rVert_p^p
			=\int \big\{\E\left(\varphi|f^{-n}\cF_0\right)\big\}\, \left(\psi\circ f^n\right)\, d\mu
			\\ & 
=\int\E\big\{\varphi\,(\psi\circ f^n)|f^{-n}\cF_0\big\}\,d\mu
=\int \varphi\left(\psi\circ f^n\right)\, d\mu
		\end{aligned}
\end{equation*}
Since $\int \varphi\,d\mu=0$, we have shown that
\[
\left\lVert\E\left(\varphi\circ f^{-n}|\cF_0\right)\right\rVert_p^p
		= \rho_{\varphi,\psi}(n)
			\leqslant C\left\lVert\psi\right\rVert_\infty r(n)
			\leqslant C\left\lVert\varphi\right\rVert_\infty^{p-1} r(n),
\]
	which yields part~\ref{itA}. 

Next, we recall by definition of conditional expectation that
$$
\E(\varphi|f^n\cF_0)=\sum_{W\in f^n\cW}1_W\, \mu(W)^{-1}\int_W \varphi\,d\mu.
$$
Hence
\begin{align*}
\big|\E(\varphi|f^n\cF_0)-\varphi\big|(x) &
= \sum_{W\in f^n\cW}1_W(x)\, \Big| \mu(W)^{-1}\int_W \varphi(y)\,d\mu(y)-\varphi(x)\Big|
\\ & \leqslant \sum_{W\in f^n\cW}1_W(x) \mu(W)^{-1}\int_W |\varphi(y)-\varphi(x)|\,d\mu(y)
\\ & \leqslant \sum_{W\in f^n\cW}1_W(x) \diam \phi(W).
\end{align*}
It follows that
\[
\lVert\E(\varphi|f^n\cF_0)-\varphi\rVert_1
\leqslant \sum_{W\in f^n\cW}\mu(W) \diam \phi(W)\leqslant Cr(n).
\]
	We deduce that
	\begin{equation*}
		\begin{aligned}
	\left\lVert\E\left(\varphi\circ f^n|\cF_0\right)-\varphi\circ f^n\right\rVert_p^p
&=\left\lVert\E\left(\varphi|f^n\cF_0\right)-\varphi\right\rVert_p^p
\\ & \leqslant \left(2\left\lVert\varphi\right\rVert_\infty\right)^{p-1}
\lVert\E(\varphi|f^n\cF_0)-\varphi\rVert_1
\leqslant \left(2\left\lVert\varphi\right\rVert_\infty\right)^{p-1}Cr(n)
		\end{aligned}
	\end{equation*}
	which establishes part~\ref{itB}.
\end{proof}

	\begin{proof}[Proof of Theorem~\ref{thm:mld}:]
The hypotheses of Theorem~\ref{ld} follow from those
of Theorem~\ref{thm:mld} by Theorem~\ref{main}.
Hence the result follows from Theorem~\ref{ld}.
\end{proof}

\section{Partially Hyperbolic Systems with Contracting Directions}\label{sec:Young}

In this section, we apply Theorem~\ref{thm:mld} together with~\cite[Theorem~A]{AM24}
to deduce the existence of a Young structure with specific return-time tails, thereby proving Theorem~\ref{thm:young}.

We begin by proving a general result deducing maximal large deviations from decay of correlations for partially hyperbolic dynamical systems.

Let $f\colon M\to M$ be a $C^1$ diffeomorphism defined on a finite-dimensional Riemannian manifold $(M,d)$ and let $X\subset M$ be a compact invariant partially hyperbolic set 
with splitting~\eqref{eq:split}.
Let $\cF_0$ be the $\sigma$-algebra generated by $\cW^s$.

\begin{prop} \label{prop:PH}
Let $\mu$ be an $f$-invariant ergodic probability measure on $X$ and 
let $r,\,r'$ be as in Definition~\ref{defn:r}.
 Let $\varphi\colon M\to\R$ be a $C^\eta$ mean zero observable.
                Assume that there exists a constant $ C > 0 $ such that
                $$\rho_{\varphi,\psi}(n)\leqslant C\lVert\psi\rVert_\infty r(n), \quad \text{for all $\psi\in L^\infty(\cF_0)$, $n\geqslant1$}.$$ 
Then, there exists a constant $C'>0$ such that
                $$
                \mld(\varphi,\varepsilon,n)\leqslant C'r'(\varepsilon,n),
                \quad\text{for all $n\geqslant1$, $\varepsilon>0$}.
                $$
\end{prop}

\begin{proof}
Let $W^s(x)$ denote the stable leaf containing $x\in X$. Then $fW^s(x)\subset W^s(fx)$ and it follows that $f^{-1}W=\bigcup_{x\in f^{-1}W}W^s(x)$ for all $W\in\cW^s$. Hence the condition $f^{-1}\cF_0\subset\cF_0$ is satisfied.

We are now in a position to apply Theorem~\ref{thm:mld}.
Condition~(i) in Theorem~\ref{thm:mld} is satisfied by assumption so it remains to verify condition~(ii).

The stable lamination is exponentially contracting:
there are constants $C>0$, $0<\lambda_1{<1}$ such that
\[
d(f^nx,f^ny)\leqslant C\lambda_1^n d(x,y),
\quad\text{for all $x,y\in X$ with $y\in W^s(x)$}.
\]
Let $\lambda_2=\lambda_1^\eta$. Then
\[
d(\varphi(f^nx),\varphi(f^ny))\leqslant \lVert\varphi\rVert_{C^\eta}(C\diam M)^\eta \lambda_2^n ,
\quad\text{for all $x,y\in X$ with $y\in W^s(x)$}.
\]
That is,
\[
\diam \varphi(f^nW)\leqslant \lVert\varphi\rVert_{C^\eta}(C\diam M)^\eta\lambda_2^n,
\quad\text{for all $W\in \cW^s$}.
\]
In particular, $\sum_{W\in f^n\cW^s}\mu(W)\diam \varphi(W)$ decays exponentially, which certainly implies the subexponential contraction rate in 
condition~(ii) in Theorem~\ref{thm:mld}.
This completes the proof.
\end{proof}

\begin{proof}[Proof of Theorem~\ref{thm:young}]
We are in the situation of~\cite[Theorem~A]{AM24}, from which it suffices to prove 
maximal large deviations rates as was done in Proposition~\ref{prop:PH}.
(In fact, to apply~\cite[Theorem~A]{AM24}, we only require the estimates for $\varepsilon=1$.)
\end{proof}

\section{MLD for Young towers with subexponential expansion/contraction}
\label{sec:Y99}

In this section, we prove Theorem~\ref{thm:Y99}.
Suppose that $f\colon M\to M$ is an invertible dynamical system with a Young structure and return map $F=f^R\colon Y\to Y$ as in Appendix~\ref{sec:towers}.
Let $(\fD,\Delta,\mu_\Delta)$ and
$(\bfD,\bDelta,\bmu_\Delta)$ be the associated two-sided and quotient Young towers.
Recall the notions of dynamically H\"older observables on $M$ and $\Delta$ and $d_\theta$-Lipschitz observables on $\bDelta$ from the Appendix.

\begin{thm} \label{thm:bDelta}
Assume that $\gcd\{R_i\}=1$ so that $\bmu_\Delta$ is mixing.
\begin{itemize}
\item[(a)] Suppose that $\bmu_\Delta(R>n)\lesssim e^{-\tau n^\omega}$ for some $\tau>0$, $0<\omega\leqslant 1$.
Then there exist $\tau'\in(0,\tau)$ and $C>0$ such that
\[
\rho_{\bphi,\bpsi}(n)\leqslant 
C\lVert\bphi\rVert_\theta \lVert\bphi\rVert_\infty e^{-\tau' n^\omega}
\quad\text{for all $d_\theta$-Lipschitz $\bphi\colon\bDelta\to\R$, $\bpsi\in L^\infty(\bDelta)$, $n\geqslant1$}.
\]
\item[(b)] Suppose that $\bmu_\Delta(R>n)\lesssim n^{-(\beta+1)}$ for some $\beta>0$.
Then there exist $C>0$ such that
\[
\rho_{\bphi,\bpsi}(n)\leqslant 
C\lVert\bphi\rVert_\theta \lVert\bphi\rVert_\infty n^{-\beta}
\quad\text{for all $d_\theta$-Lipschitz $\bphi\colon\bDelta\to\R$, $\bpsi\in L^\infty(\bDelta)$, $n\geqslant1$}.
\]
\end{itemize}
\end{thm}

\begin{proof}
See for example~\cite[Theorem~2.7(ii)]{KKM19} for part~(a).
Part~(b) follows from~\cite{Y99} or~\cite[Theorem~2.7(i)]{KKM19}.
\end{proof}

Let $\cF_0$ be the $\sigma$-algebra on $\Delta$ generated by sets of the form
$(\bar\pi^{-1}E)\times\{\ell\}$ where $E$ is a measurable subset of $\bY_k$, $k\ge1$ and $0\leqslant\ell<R|\bY_k$.
Then $\fD^{-1}\cF_0\subset\cF_0$.
Let $L^\infty(\cF_0)$ denote the space of observables $\psi\in L^\infty(\Delta)$ such that $\psi$ is $\cF_0$ measurable.

\begin{lemma} \label{lem:DMN}
Fix a sequence $r(n)>0$.
Suppose that $\rho_{\bar\varphi,\bar\psi}(n)\leqslant 
\lVert\bphi\rVert_\theta \lVert\bpsi\rVert_\infty r(n)$ for 
all $d_\theta$-Lipschitz $\bphi\colon\bDelta\to\R$ and all $\bpsi\in L^\infty(\bar\Delta)$.
Then for any dynamically H\"older $\varphi\colon\Delta\to\R$,
there exists $C>0$ such that
\[
\rho_{\varphi,\psi}(n)\leqslant C\lVert\psi\rVert_\infty r(n)
\quad\text{for all $\psi\in L^\infty(\cF_0)$ and $n\geqslant1$}.
\]
\end{lemma}

\begin{proof}
This is identical to the argument in~\cite[Section~3.2]{DMN20}.
For completeness, we sketch the main steps.

Let $L\colon L^1(\bDelta)\to L^1(\bDelta)$ denote the transfer operator corresponding to $\bar f\colon\bDelta\to\bDelta$.
It follows for instance from~\cite[Proposition~5.3]{KKM19} that we can choose a sequence of observables $\varphi_\ell\colon \Delta\to\R$ such that
\begin{itemize}
\item[(i)] $\varphi_\ell$ is $\cF_0$-measurable and hence projects to an observable $\bphi_\ell\colon \bDelta\to\R$;
\item[(ii)] $\sup_{\ell\geqslant1}\lVert L^\ell\bphi_\ell\rVert_\theta<\infty$;
\item[(iii)] $\lim_{\ell\to\infty}\lVert\varphi\circ \fD^\ell-\varphi_\ell\rVert_1=0$.
\end{itemize}

Let $\psi\in L^\infty(\cF_0)$ with projection 
$\bpsi\in L^\infty(\bDelta)$.
Following~\cite[proof of Corollary~5.4]{KKM19},
\[
\rho_{ \varphi ,\psi}(n)\leqslant |I_1(\ell,n)|+|I_2(\ell,n)|+I_3(\ell,n),
\]
where
\begin{align*}
I_1(\ell,n) & = \int_\Delta (\varphi\circ \fD^\ell-\varphi_\ell)\,\psi\circ \fD^{\ell+n}\,d\mu_\Delta, 
\\
I_2(\ell,n) & = \int_\Delta (\varphi_\ell-\varphi\circ \fD^\ell)\,d\mu_\Delta \; 
\int_\Delta \psi \,d\mu_\Delta ,
\qquad 
I_3(\ell,n) = \rho_{ \varphi_\ell,\psi}(\ell+n) .
\end{align*}
By (iii), $\lim_{\ell\to\infty}I_j(\ell,n)=0$ uniformly in $n$ for $j=1$ and $j=2$.
By (i) and the main hypothesis,
\[
I_3(\ell,n)=
\rho_{\bphi_\ell, \bpsi}(\ell+n)=
\rho_{L^\ell\bphi_\ell, \bpsi}(n)
\leqslant \lVert L^\ell \bphi_\ell \rVert_\theta \lVert \bpsi \rVert_\infty r(n)= \lVert L^\ell \bphi_\ell \rVert_\theta \lVert \psi \rVert_\infty r(n).
\]
By (ii), $\sup_{\ell\geqslant 1} |I_3(\ell,n)|\lesssim r(n)$
and the result follows.
\end{proof}

\begin{cor} \label{cor:Delta}
Let $\varphi:\Delta\to\R$ be dynamically H\"older.
\begin{itemize}
\item[(a)] Suppose that $\mu_\Delta(R>n)\lesssim e^{-\tau n^\omega}$ for some $\tau>0$, $0<\omega\leqslant 1$.
Then there exist $\tau'\in(0,\tau)$ and $C>0$ such that
\[
\rho_{\varphi,\psi}(n)\leqslant 
C \lVert\psi\rVert_\infty e^{-\tau' n^\omega}
\quad\text{for all $\psi\in L^\infty(\cF_0)$, $n\geqslant1$}.
\]
\item[(b)] Suppose that $\mu_\Delta(R>n)\lesssim n^{-(\beta+1)}$ for some $\beta>0$.
Then there exist $C>0$ such that
\[
\rho_{\varphi,\psi}(n)\leqslant 
C\lVert\psi\rVert_\infty n^{-\beta}
\quad\text{for all $\psi\in L^\infty(\cF_0)$, $n\geqslant1$}.
\]
\end{itemize}
\end{cor}

\begin{proof} This is immediate from Theorem~\ref{thm:bDelta} and Lemma~\ref{lem:DMN}.
\end{proof}

\bibliographystyle{acm}

\begin{proof}[Proof of Theorem~\ref{thm:Y99}]
Since $\pi_\Delta\colon\Delta\to M$ is a measure-preserving semiconjugacy and H\"older observables $\varphi\colon M\to\R$ lift to dynamically H\"older observables $\varphi\circ\pi_\Delta\colon\Delta\to\R$, it suffices to prove MLD for dynamically H\"older observables $\varphi\colon\Delta\to\R$.

We begin by proving the result under the additional assumption that $\gcd\{R_i\}=1$ (so $\mu_\Delta$ and $\bmu_\Delta$ are mixing).
First, consider part (b) so $\mu_\Delta(R>n)\lesssim n^{-(\beta+1)}$.
Let $r(n)=n^{-\beta}$.
By Corollary~\ref{cor:Delta}(b), condition~(i) of Theorem~\ref{thm:mld} is satisfied.

Define $h_n(x)=\#\{0\leqslant j\leqslant n: \fD^jx\in Y\}$ for $x\in \Delta$.
Then $\diam \fD^n W^s(x)\lesssim \gamma^{h_n(x)}$ and hence
$\diam \varphi(\fD^n W^s(x))\lesssim \gamma^{h_n(x)}$.
By~\cite[Lemma~3.2]{DMN20}, condition~(ii) of Theorem~\ref{thm:mld} is satisfied.
Hence the result follows from Theorem~\ref{thm:mld}.

Next, consider part (a) so $\mu_\Delta(R>n)\lesssim e^{-\tau n^{\omega}}$.
By Corollary~\ref{cor:Delta}, condition~(i) of Theorem~\ref{thm:mld} is satisfied with $r(n)=e^{-\tau' n^{\omega}}$.

As in the proof of~\cite[Lemma 3.2]{DMN20},
\[
\sum_{W\in f^n\cW^s}\mu(W)\diam\varphi(W)
\lesssim \sum_{k=1}^{n+1}k\gamma^k
\Big(n\mu(R\geqslant n/k)+\int_Y 1_{\{R>n\}}R\,d\mu\Big).
\]
Now, $\int_Y 1_{\{R>n\}}R\,d\mu=\sum_{j>n} j\mu(R=j)\lesssim e^{-\tau' n^\omega}$ for all $\tau'\in(0,\tau)$.
Hence
\[
\sum_{k=1}^{n+1}k\gamma^k
\int_Y 1_{\{R>n\}}R\,d\mu \lesssim e^{-\tau' n^\omega}.
\]
Next, choose $\omega'\in(0,\omega)$ and write 
\begin{align*}
\sum_{k=1}^{n+1}k\gamma^k \mu(R\geqslant n/k) & =
\sum_{1\leqslant k\leqslant n^{\omega'}}k\gamma^k \mu(R\geqslant n/k)
+\sum_{n^{\omega'} <k\leqslant n}k\gamma^k \mu(R\geqslant n/k)
\\ & \lesssim
\mu(R\geqslant n^{1-\omega'})
+\sum_{k>n^{\omega'}}k\gamma^k
 \lesssim
n e^{-\tau n^{(1-\omega')\omega}} +e^{-\tau' n^{\omega'}}
\end{align*}
for some $\tau'$.
Taking $\omega'=\omega/(1+\omega)$, we obtain
\[
\sum_{W\in f^n\cW^s}\mu(W)\diam\varphi(W) \lesssim e^{-\tau' n^{\omega/(1+\omega)}},
\]
so condition~(ii) of 
Theorem~\ref{thm:mld} is satisfied with $r(n)=e^{-\tau' n^{\omega/(1+\omega)}}$.
Again, the result follows from Theorem~\ref{thm:mld}.

It remains to relax the assumption that $\gcd\{R_i\}=1$. 
Suppose that 
$\gcd\{R_i\}=s\geqslant 2$. 
We decompose the tower into $\Delta=\bigcup_{r=1}^s\Delta^{(r)}$ 
where each $\Delta^{(r)}$ is an $f_\Delta^s$-invariant tower over $Y$ with return time $R^{(r)}=\frac1s R$.
Note that $f_\Delta$ cyclicly permutes $\Delta^{(1)},\dots,\Delta^{(s)}$.
Moreover, we obtain mixing $f_\Delta^s$-invariant probability measures 
$\mu_\Delta^{(r)}=s\mu_\Delta|\Delta^{(r)}$ on $\Delta^{(r)}$.

Let $\varphi_k=\sum_{j=0}^{k-1}\varphi\circ f_\Delta^j$. 
For $E\subset\Delta$ measurable, 
$\mu_\Delta(E)=\frac1s\sum_{r=1}^s\mu_\Delta^{(r)}(E\cap\Delta^{(r)})$.
Hence,
\begin{equation} \label{eq:E}
\mld(\varphi,\varepsilon,n) = \frac1s\sum_{r=1}^s \mu_\Delta^{(r)}
\Big\{x\in\Delta^{(r)}\sup_{k\geqslant n}\Big|\frac1k\varphi_k(x) -\int_\Delta\varphi\,d\mu_\Delta\Big|>\varepsilon\Big\}.
\end{equation}

Next, we note that
$\varphi_k=\sum_{j=0}^{[k/s]}\varphi_s\circ f_\Delta^{sj}+O(1)$
and
$\int_{\Delta^{(r)}}\varphi_s\,d\mu_\Delta^{(r)}
=s\int_\Delta \varphi\,d\mu_\Delta$ for $r=1,\dots,s$.
It follows that
on $\Delta^{(r)}$,
\[
\frac1k\varphi_k -\int_\Delta\varphi\,d\mu_\Delta
=\frac1s\Big\{\frac{1}{[k/s]}\sum_{j=0}^{[k/s]-1}\varphi_s\circ f_\Delta^{sj}
-\int_{\Delta^{(r)}}\varphi_s\,d\mu_\Delta^{(r)}\Big\}+ O\Big(\frac1k\Big),
\]
and hence
\begin{equation} \label{eq:gcd}
\sup_{k\geqslant n}\Big|\frac1k\varphi_k -\int_\Delta\varphi\,d\mu_\Delta\Big|
= \frac1s\sup_{k\geqslant [n/s]}\Big|\frac{1}{k}\sum_{j=0}^{k-1}\varphi_s\circ f_\Delta^{sj}
-\int_{\Delta^{(r)}}\varphi_s\,d\mu_\Delta^{(r)}\Big|+ O\Big(\frac1n\Big).
\end{equation}

By~\eqref{eq:E} and~\eqref{eq:gcd}, MLD for $\varphi$ on $(f_\Delta,\Delta,\mu_\Delta)$ reduces to MLD for $\varphi_s$ on $(f_\Delta^s,\Delta^{(r)},\mu_\Delta^{(r)})$.
Note that $\varphi_s$ is dynamically H\"older whenever $\varphi$ is dynamically H\"older. Since $(f_\Delta^s,\Delta^{(r)},\mu_\Delta^{(r)})$ is mixing, we have reduced to the mixing case $\gcd\{R_i\}=1$.
\end{proof}

\section{Examples: slowly mixing billiards} \label{sec:ex}

In this section, we use Theorem~\ref{thm:Y99} to prove MLD for various slowly mixing billiard examples.
For background material on billiards, we refer to~\cite{ChernovMarkarian}.
The billiard domain, denoted by $Q$, is a compact connected subset of $\R^2$ or $\T^2$ with piecewise smooth boundary and the billiard flow $f_t$ is defined on $Q\times S^1$. Fix a point $q\in Q$ and a unit vector $\psi\in S^1$.
Then $q$ moves in straight lines with unit speed in direction $\psi$ until reflecting (angle of reflection equalling the angle of incidence) off the boundary $\partial Q$. This defines a volume-preserving flow.
A natural Poincar\'e section is given by $M=\partial Q\times[-\pi/2,\pi/2]$ corresponding to collisions with $\partial Q$ (with outgoing velocities in $[-\pi/2,\pi/2]$).  The Poincar\'e map $f\colon M\to M$ is called the {\em collision map} or the {\em billiard map}.  It preserves a probability measure $\mu$, equivalent to Lebesgue, called Liouville measure.

\begin{example}[Bunimovich stadia~\cite{Bunimovich79}]
These are convex billiard domains $Q\subset\R^2$ where $\partial Q$
is a simple closed curve consisting of
two semicircles $C_1,\,C_2$ of radius $1$ and
two parallel line segments $S_1,\,S_2$ of length $L$ tangent to the semicircles.

By~\cite{Markarian04,CZ08},
the billiard map $f\colon M\to M$ is modelled by a Young tower with polynomial tails $n^{-2}$.
By Theorem~\ref{thm:Y99}, we obtain $\mld(\varphi, \varepsilon, n)\lesssim n^{-1}$ for dynamically H\"older observables $\varphi$.
\end{example}

\begin{example}[Semi-dispersing billiards]
These are billiard domains 
$Q= R\setminus\Omega$ where $R\subset\R^2$ is a rectangle and $\Omega\subset \Int R$ is a disjoint union of strictly convex regions with $C^3$ boundaries and nonvanishing curvature. 

By~\cite{CZ05a,CZ08},
the billiard map $f\colon M\to M$ is modelled by a Young tower with polynomial tails $n^{-2}$.
By Theorem~\ref{thm:Y99}, we obtain $\mld(\varphi, \varepsilon, n)\lesssim n^{-1}$ for dynamically H\"older observables $\varphi$.
\end{example}

\begin{example}[Billiards with cusps and flat cusps]
These are billiard domains $Q\subset\R^2$ where $\partial Q$ is a simple closed curve consisting of finitely many convex inwards $C^3$ curves with nonvanishing curvature such that the interior angles at corner points are zero.
By~\cite{ChernovMarkarian07,CZ08}, the billiard map $f\colon M\to M$ is modelled by a Young tower with polynomial tails $n^{-2}$.
By Theorem~\ref{thm:Y99}, we obtain $\mld(\varphi, \varepsilon, n)\lesssim n^{-1}$ for dynamically H\"older observables $\varphi$.

In~\cite{Z17}, the nonvanishing curvature assumption was relaxed and it was shown that billiards with flat cusps are modelled by Young towers with polynomial tails $n^{-(\beta+1)}$ where $\beta\in(0,1)$. 
By Theorem~\ref{thm:Y99}, we obtain $\mld(\varphi, \varepsilon, n)\lesssim n^{-\beta}$ for dynamically H\"older observables $\varphi$.
\end{example}

\begin{example}[Bunimovich flowers]
This is a class of billiards where $\partial Q$ is a union of smooth curves that are either convex inwards (with bounded nonvanishing curvature) or arcs contained in semicircles, satisfying some additional technical assumptions, see~\cite{Bunimovich73,CZ05a}.

By~\cite{CZ05a,KB25},
the billiard map $f\colon M\to M$ is modelled by a Young tower with polynomial tails $n^{-3}$.
By Theorem~\ref{thm:Y99}, we obtain $\mld(\varphi, \varepsilon, n)\lesssim n^{-2}$ for dynamically H\"older observables $\varphi$.
\end{example}

\begin{example}[Dispersing billiards with flat points]
Classical dispersing billiards have $Q=\T^2\setminus\Omega$ where $\Omega$ is a disjoint union of strictly convex regions with $C^3$ boundaries and nonvanishing curvature. By~\cite{Y98,Chernov99}, the billiard maps have exponential decay of correlations. 

The nonvanishing curvature condition is relaxed in~\cite{CZ05}: they allowed two flat points with a periodic orbit running between the flat points. The boundary near the flat points in~\cite{CZ05} has profile $\pm(1+|x|^b)$ for some $b>2$.
By~\cite{CZ05,KB25}, these billiard maps are modelled by Young towers with polynomial tails $n^{-(\beta+1)}$ where $\beta=(b+2)/(b-2)\in(1,\infty)$.
By Theorem~\ref{thm:Y99}, we obtain $\mld(\varphi, \varepsilon, n)\lesssim n^{-\beta}$ for dynamically H\"older observables $\varphi$.
\end{example}

\appendix
\section{Background material on Young towers}
\label{sec:towers}

\subsection{Young structures}
\label{sec:Ystr}

Consider $M$ to be a finite-dimensional compact Riemannian manifold, and $f \colon M \to M$ a piecewise $C^1$ diffeomorphism, possibly with critical, singular, or discontinuity points. Let $d$ denote the distance on $M$, and let $\leb$ represent the Lebesgue measure on the Borel subsets of $M$, both induced by the Riemannian metric. Given a submanifold $\gamma \subset M$, we denote by $\leb_\gamma$ the Lebesgue measure on $\gamma$, induced by the restriction of the Riemannian metric to $\gamma$. 

We say that $\Gamma$ is a \emph{continuous family} of $C^1$ disks in $M$ with $\dim\Gamma=k$ if there are a compact metric space $K$, a unit disk $D\subset\R^k$ for some $1\leqslant k \leqslant \dim M-1$ and an injective continuous function $\Phi\colon K\times D\to M$ such that
\begin{itemize}
	\item $\Gamma=\left\{\Phi\left(\{x\}\times D\right)\colon x\in K\right\}$;
	\item$\Phi$ maps $K\times D$ homeomorphically onto its image;
	\item $x\mapsto\Phi|_{\{x\}\times D}$ defines a continuous map from $K$ into the space of $C^1$ embeddings of $D$ into $M$. 
\end{itemize}

We say that a compact set $Y \subset M$ has a \emph{product structure} if there exist continuous families of $C^1$ disks $\Gamma^s$ (stable disks) and $\Gamma^u$ (unstable disks), such that
\begin{itemize}
	\item $Y = \left(\bigcup_{\gamma \in \Gamma^s} \gamma \right) \cap \left(\bigcup_{\gamma \in \Gamma^u} \gamma \right)$;
	\item $\dim \Gamma^s + \dim \Gamma^u = \dim M$;
	\item each $\gamma \in \Gamma^s$ meets each $\gamma \in \Gamma^u$ in exactly one point.
\end{itemize}
Let $\gamma^*(y)$ denote the disk in $\Gamma^*$ containing $y \in Y$, for $* = s,u$. 

We say that $Y_0 \subset Y$ is an \emph{$s$-subset} if $Y_0$ has a product structure with respect to families $\Gamma_0^s$ and $\Gamma_0^u$ such that $\Gamma_0^s \subset \Gamma^s$ and $\Gamma_0^u = \Gamma^u$; \emph{$u$-subsets} are defined similarly.

The compact set $Y$ has a \emph{Young structure} or \emph{two-sided Gibbs-Markov structure} if $Y$ has a product structure such that
\[
\leb_\gamma \big( Y \cap \gamma \big) > 0, \quad \text{for all } \gamma \in \Gamma^u,
\]
and conditions \ref{Y1}--\ref{Y5} below are satisfied.

\begin{enumerate}[label=\textbf{(Y$_{\arabic*}$)}]
	\item \label{Y1} \textbf{Markov:} there are pairwise disjoint $s$-subsets $Y_1,\, Y_2, \cdots\subset Y$ such that
	\begin{itemize}
		\item $\leb_\gamma\left(\left(Y \setminus \cup_iY_i\right)\cap\gamma\right)=0$, for all $\gamma\in\Gamma^u$;
		\item for each $i\geqslant1$, there is $R_i\in\N$ such that $f^{R_i}\left(Y_i\right)$ is a $u$-subset and, moreover, for all $y\in Y_i$,
		\begin{equation*}
			f^{R_i}\left(\gamma^s(y)\right)\subset \gamma^s\left(f^{R_i}y\right) \text{ and } 	f^{R_i}\left(\gamma^u(y)\right)\supset \gamma^u\left(f^{R_i}y\right).
		\end{equation*}
	\end{itemize}
\end{enumerate}

The \emph{recurrence time} $R \colon Y \to \N$ and \emph{return map} $F=f^R \colon Y \to Y$ are defined to be
\begin{equation*}
	R|_{Y_i} = R_i \quad \text{and} \quad F|_{Y_i} = f^{R_i}|_{Y_i}, \quad
i \geqslant 1.
\end{equation*}
We remark that $R$ and $F$ are defined on a full $\leb_\gamma$-measure subset of $Y \cap \gamma$, for each $\gamma \in \Gamma^u$. Thus, there exists a set $Y' \subset Y$ intersecting each $\gamma \in \Gamma^u$ in a full $\leb_\gamma$-measure subset, such that $F^ny$ belongs to some $Y_i$, for all $n \geqslant 0$ and $y \in Y'$. 
For points $y, y' \in Y'$, we define the \emph{separation time}
\begin{equation}\label{eq.separa}
	s(y,y') = \min \left\{ n \geqslant 0 \colon F^ny \text{ and } F^ny' \text{ lie in distinct } Y_i \text{'s} \right\},
\end{equation}
with the convention that $\min\emptyset = \infty$. For definiteness, we set the separation time equal to zero for all other points. 

For the remaining conditions, we consider constants $C \geqslant1$ and $0 < \beta < 1$ depending only on $f$ and $Y$.

\begin{enumerate}[label=\textbf{(Y$_{\arabic*}$)}]
	\setcounter{enumi}{1}
	\item \label{Y2} \textbf{Contraction on stable disks:} for all $i\geqslant1$, $\gamma\in\Gamma^s$ and $y,y'\in\gamma\cap Y_i$,
	\begin{itemize}
		\item $d\left(F^ny,F^ny'\right)\leqslant C\beta^n$, for all $n\geqslant0$;
		\item$d\left(f^jy,f^jy'\right)\leqslant Cd(y,y')$, for all $1\leqslant j \leqslant R_i$.	
	\end{itemize}
	
	\item \label{Y3} \textbf{Expansion on unstable disks:} for all $i\geqslant1$, $\gamma\in\Gamma^u$ and $y,y'\in\gamma\cap Y_i$,
	\begin{itemize}
		\item $d\left(F^ny,F^ny'\right)\leqslant C\beta^{s(y,y')-n}$, for all $n\geqslant0$;
		\item$d\left(f^jy,f^jy'\right)\leqslant Cd\left(Fy,Fy'\right)$, for all $1\leqslant j \leqslant R_i$.	
	\end{itemize}

	\item \label{Y4}\textbf{ Bounded distortion:} for all $i\geqslant 1$, $\gamma\in\Gamma^u$ and $y,y'\in\gamma\cap Y_i$,
	\begin{equation*}
		\log\frac{\det DF|_{T_y\gamma}}{\det DF|_{T_{y'}\gamma}}\leqslant C\beta^{s\left(Fy,Fy'\right)}.
	\end{equation*}
\end{enumerate}
	
Let $\gamma,\,\gamma'\in\Gamma^u$.  The
\emph{stable holonomy map} $\Theta_{\gamma,\gamma'} \colon \gamma \cap Y \to \gamma' \cap Y$ is given by
\[
\Theta_{\gamma,\gamma'}(y) = \gamma^s(y) \cap \gamma',\quad
y \in \gamma \cap Y.
\]
\begin{enumerate}[label=\textbf{(Y$_{\arabic*}$)}]
	\setcounter{enumi}{4}
	\item \label{Y5}
\textbf{ Regularity of the stable holonomy:} for all $\gamma,\gamma'\in\Gamma^u$, the measure $\left(\Theta_{\gamma,\gamma'}\right)_* \leb_\gamma$ is absolutely continuous with respect to $\leb_{\gamma'}$ and its density $\rho_{\gamma,\gamma'}$ satisfies\begin{equation*}
		C^{-1}\leqslant\int_{\gamma'\cap Y}\rho_{\gamma,\gamma'} d\leb_{\gamma'}\leqslant C \quad\text{and}\quad\log\frac{\rho_{\gamma,\gamma'}(x)}{\rho_{\gamma,\gamma'}(y)}\leqslant C\beta^{s(y,y')},
	\end{equation*}
	for all $y,y'\in\gamma'\cap Y$.
\end{enumerate}
Finally, we assume that $\int_{Y\cap \gamma}R\,dm_\gamma<\infty$ for some, and hence all, $\gamma\in\Gamma^u$.

By Young~\cite[Proof of Theorem~1]{Y98}, there is an $F$-invariant Borel probability measure $\mu_Y$ on $Y$ with 
absolutely continuous conditional measures on unstable disks $\gamma^u$.

\subsection{Two-sided Young towers}
\label{sec:twosided}

We define the \emph{two-sided Young tower}
$\Delta=\{(y,\ell):y\in Y,\,0\le \ell<R(y)\}$ and tower map
\[
\fD\colon \Delta\to \Delta\,, \qquad
\fD(y,\ell)=\begin{cases} (y,\ell+1) , &0\le \ell<R(y)-1 \\
(Fy,0) , &\ell=R(y)-1 
\end{cases}\,,
\]
with ergodic invariant probability measure $\mu_\Delta=(\mu_Y\times{\rm counting})/\int_Y R\,d\mu_Y$.
Define the semiconjugacy
\[
\pi_\Delta:\Delta\to M, \qquad \pi_\Delta(y,\ell)=f^\ell y.
\]
Then $\mu=\pi_\Delta^*\mu_\Delta$ is an $f$-invariant probability measure on $M$.
By Young~\cite[Theorem~1]{Y98}, 
$\mu$ is an SRB measure: $f$ has no zero Lyapunov exponents $\mu$-a.e.\ and $\mu$ induces absolutely continuous conditional measures on unstable disks $\gamma^u$.
In addition, $\mu$ is mixing if and only if $\gcd \{R_i\}=1$.

\subsection{Quotient towers}

We define an equivalence relation $\sim$ on $Y$ and on $\Delta$ where
$(y,\ell)\sim(y',\ell)$ if $y'\in \gamma^s(y)$.
Note that $R:Y\to\Z^+$ is constant on equivalence classes.
Let $\bY=Y/\sim$.
Define the \emph{quotient Young tower}
$\bDelta=\{(y,\ell):y\in \bY,\,0\le \ell<R(y)\}=\Delta/\sim$ and tower map
\[
\bfD\colon \bDelta\to \bDelta\,, \qquad
\bfD(y,\ell)=\begin{cases} (y,\ell+1) , &0\le \ell<R(y)-1 \\
(\bar Fy,0) , &\ell=R(y)-1 
\end{cases}\,,
\]
with ergodic invariant probability measure $\bmu_\Delta=(\mu_\bY\times{\rm counting})/\int_\bY R\,d\bmu_Y$.
The natural projection $\bpi\colon \Delta\to\bDelta$ defines a measure-preserving semiconjugacy.
Again, $\bmu$ is mixing if and only if $\gcd \{R_i\}=1$.

\subsection{Observables}
The separation time $s$ on $\Delta$ is constant on equivalence classes and hence defines a separation time $s$ on $\bDelta$.
Then $d_\theta(y,y')=\theta^{s(y,y')}$ is a metric on $\bDelta$ for $\theta\in(0,1)$.
Given an observable $\bphi:\bDelta\to\R$, define
\[
\lVert\bphi\rVert_\theta=
\lVert\bphi\rVert_\infty+\sup_{\stackrel{(y,\ell),\,(y',\ell)\in\bDelta}{y\neq y'}}
\frac{|\bphi(y,\ell)-\bphi(y',\ell)|}{d_\theta(y,y')}.
\]
We say that $\bphi$ is \emph{$d_\theta$-Lipschitz} if $\lVert\bphi\rVert_\theta<\infty$.

Returning to the two-sided tower $\Delta$, we say that
an observable $\varphi:\Delta\to\R$ is \emph{dynamically H\"older} if it is bounded and there exist constants $C>0$, $\theta\in(0,1)$ such that
\begin{equation} \label{eq:dyn}
|\varphi(y,\ell)-\varphi(y',\ell)|\le C\big(d_Y(y,y')+d_\theta(\bar\pi y,\bar\pi y')\big)
\end{equation}
for all $(y,\ell),\,(y',\ell)\in \Delta$.

For the underlying system, an observable $\varphi\colon M\to\R$ is \emph{dynamically H\"older} if the lifted observable $\varphi\circ \pi_\Delta\colon \Delta\to\R$ is dynamically H\"older. By definition of Young structure, H\"older observables on $M$ are automatically dynamically H\"older (with $\theta=\beta^\eta$ where $\eta$ is the H\"older exponent).

\end{document}